\pdfoutput=1
\documentclass[11pt,reqno]{amsart}

\usepackage[dvipsnames]{xcolor}
\usepackage{mathtools}
\usepackage{amssymb}
\usepackage[scr=boondox]{mathalpha}
\usepackage[low-sup]{subdepth}
\usepackage{tikz-cd}
\usepackage{amsthm}
\usepackage{hyperref}
\usepackage[capitalize]{cleveref}

\mathtoolsset{mathic=true}

\tikzcdset{
  arrow style=tikz,
  diagrams={>={Computer Modern Rightarrow}}
  }

\definecolor{dark-red}{rgb}{0.6,0,0}
\definecolor{dark-green}{rgb}{0,0.4,0}
\definecolor{nice-purple}{HTML}{8A0087}
\hypersetup{
  colorlinks,
  citecolor=dark-green,
  linkcolor=dark-red,
  urlcolor=nice-purple,
  pdfauthor={Matthew Bertucci},
  pdftitle={Taylor conditions over finite fields},
  }

\ifdefined\newcounteralias
  \newcommand{\newaliastheorem}[3]{%
    \newtheorem{#1}[#2]{#3}%
    }
\else
  \usepackage{aliascnt}
  \newcommand{\newaliastheorem}[3]{%
    \newaliascnt{#1}{#2}%
    \newtheorem{#1}[#1]{#3}%
    \aliascntresetthe{#1}%
    }
\fi
\newtheorem{claim}{Claim}

\newtheorem*{theorem*}{Theorem}
\newaliastheorem{lemma}{theorem}{Lemma}
\newaliastheorem{corollary}{theorem}{Corollary}
\newaliastheorem{proposition}{theorem}{Proposition}

\newtheorem{maintheorem}{Theorem}

\crefname{maintheorem}{Theorem}{Theorems} % match 'capitalize' option
\Crefname{maintheorem}{Theorem}{Theorems}

\theoremstyle{definition}
\newtheorem*{definition}{Definition}
\newaliastheorem{example}{theorem}{Example}
\newtheorem{question}{Question}
\crefname{question}{Question}{Questions} % match 'capitalize' option
\Crefname{question}{Question}{Questions}

\theoremstyle{remark}
\newaliastheorem{remark}{theorem}{Remark}

\counterwithin{equation}{subsection}
\crefname{equation}{}{}

\newcommand{\bA}{\mathbb{A}}
\newcommand{\bF}{\mathbb{F}}
\newcommand{\bL}{\mathbb{L}}

\newcommand{\bP}{\mathbb{P}}
\newcommand{\bV}{\mathbb{V}}
\newcommand{\calP}{\mathcal{P}}
\newcommand{\calQ}{\mathcal{Q}}
\newcommand{\calT}{\mathcal{T}}
\newcommand{\frakm}{\mathfrak{m}}
\newcommand{\high}{\mathrm{high}}
\newcommand{\homog}{\mathrm{homog}}
\newcommand{\into}{\hookrightarrow}
\newcommand{\low}{\mathrm{low}}
\newcommand{\med}{\mathrm{med}}
\newcommand{\onto}{\twoheadrightarrow}
\newcommand{\PP}{\mathscr{P}}
\newcommand{\red}{\mathrm{red}}
\newcommand{\sheaf}{\mathscr}

\DeclareMathOperator{\Der}{Der}
\DeclareMathOperator{\eval}{eval}
\DeclareMathOperator{\Hom}{Hom}

\DeclareMathOperator{\Prob}{Prob}
\DeclareMathOperator{\Spec}{Spec}
\DeclareMathOperator{\Sym}{Sym}

\DeclarePairedDelimiter\floor{\lfloor}{\rfloor}

\newcommand{\restr}[3][]{#2\csname#1\endcsname|\sb{#3}}

\title{Taylor conditions over finite fields}
\author{Matthew Bertucci}
\address{155 South 1400 East, Salt Lake City, UT 84103}
\email{bertucci@math.utah.edu}
\urladdr{https://www.math.utah.edu/~bertucci/}

\begin{document}

\begin{abstract}
We extend Poonen's Bertini theorem over finite fields to Taylor conditions arising from locally free quotients of the sheaf of differentials on projective space.
This is motivated by a result of Bilu and Howe in the motivic setting that allows for significantly more general Taylor conditions.
\end{abstract}

\maketitle

\tableofcontents

\section{Introduction}

For \(X\) a smooth quasiprojective subscheme of \(\bP^n\) over a finite field \(\bF_q\), Poonen showed in \cite{poonen:bertini} the existence of smooth hypersurface sections of \(X\) and computed the asymptotic density of smooth hypersurface sections to be \(\zeta_X(\dim X+1)^{-1}\), where \(\zeta_X\) is the zeta function of \(X\).
He also allowed for prescribing the first few coefficients of the Taylor expansions of hypersurfaces at finitely many points.
It is natural to extend the problem to more general conditions on the Taylor expansions.
As far as the author knows, questions like the following are not within the scope of Poonen's theorem or its existing generalizations.%
\footnote{These include \cite{bucur-kedlaya}, \cite{erman-wood}, \cite{ghosh-krishna}, \cite{gunther}, \cite{poonen:subvar}, and \cite{wutz:thesis}.}%

\begin{question} \label{curve-q}
Assume \(\operatorname{char}(\bF_q)\neq 2\).
Choose a finite, reduced, degree 4 subscheme \(Y\) of \(\bP^2_{\bF_q}\) whose points are geometrically in general position. 
Let \(\iota:X\into \bP^2_{\bF_q}\) be a curve whose geometric points are in general position with the points of \(Y\).
For each closed point \(x\in X\), there is a unique smooth conic \(C_x\) passing through the four points and \(x\).
What is the probability that a random plane curve intersects \(C_x\) transversely at \(x\) for each closed point \(x\in X\)?
\end{question}

Note that above, we do not assume $X$ is smooth.
To answer the question we will remove the necessity that properties of $X$ be related to Taylor conditions parameterized by $X$. See \cref{rel-smooth-app} for a concrete use-case.
The answer to \cref{curve-q} follows as a special case, explained in \cref{curve-ex}, and requires considering Taylor conditions arising from locally free quotients of the sheaf of differentials on projective space.
Such Taylor conditions are addressed in the following theorem which is the main result of this paper.
See \cref{sheaf-PP} for the definition of and important results about the sheaf of principal parts \(\PP^1\).

\begin{maintheorem} \label{subsheaf-bertini}
Let \(X\) be a quasiprojective subscheme of \(\bP^n_{\bF_q}\) of dimension \(m\) with locally closed embedding \(\iota\).
Let \(\sheaf{Q}\) be a locally free quotient of \(\iota^*\Omega^1_{\bP^n}\) of rank \(\ell\geq m\), and let \(\sheaf{K}\) denote the kernel of \(\iota^*\Omega^1_{\bP^n}\onto\sheaf{Q}\).
For each \(d\), define
  \[
    \sheaf{E}_d := \bigl(\iota^*\PP^1(\sheaf{O}_{\bP^n}(d))\bigr)/\sheaf{K}(d)
  \]
where we view \(\sheaf{K}(d)\) as a subsheaf of \(\iota^*\PP^1(\sheaf{O}_{\bP^n}(d))\) via the exact sequence in the top row of the commutative diagram
  \[
    \begin{tikzcd}[column sep=small]
    0 \ar[r] &
    \iota^*\Omega^1_{\bP^n}(d) \ar[r] \ar[d] &
    \iota^*\PP^1(\sheaf{O}_{\bP^n}(d)) \ar[r] \ar[d] &
    \sheaf{O}_X(d) \ar[r] &
    0. \\
    &
    \sheaf{Q}(d) \ar[r] &
    \sheaf{E}_d
    \end{tikzcd}
  \]
This defines a \(1\)-infinitesimal Taylor condition \(\calT_d\) on \(\bP^n\) such that at each closed point \(x\), \(\calT_{d,x}\subseteq\sheaf{O}_{\bP^n}(d)_x/\frakm_x^2\) is given by not vanishing in the fiber of \(\sheaf{E}_d\) at \(x\).
By convention, \(\calT_d\) is always satisfied if \(x\notin X\).

Define
  \[
    \calP_d := \{f\in S_d\mid f \text{ satisfies \(\calT_d\) at all closed \(x\in\bP^n\)}\}.
  \]
Then
  \[
    \lim_{d\to\infty} \Prob(f\in \calP_d)
    = \prod_{\textup{closed } x\in X} \bigl(1-q^{-(\ell+1)\deg(x)}\bigr) = \zeta_X(\ell+1)^{-1}.
  \]
\end{maintheorem}

For \(X\) smooth, taking \(\sheaf{Q}=\Omega_X^1\) recovers Poonen's Bertini theorem since in that case, the nonvanishing of a section \(f\) in the fiber \(\restr{\sheaf{E}_d}{x}=\sheaf{O}_X(d)_x/\frakm_x^2\) is exactly the condition of the hypersurface section \(H_f\cap X\) being smooth at \(x\).

As in Poonen's discussion after \cite[Theorem 1.1]{poonen:bertini}, this Euler product is the expected probability if we assume that satisfaction of the Taylor condition $\calT_d$ is independent across all $x\in X$.
Indeed, a section $f$ satisfying $\calT_d$ at $x$ amounts to $\ell+1$ linear conditions on the Taylor coefficients of $f$ over the residue field $\kappa(x)$, so the probability of satisfaction at $x$ is
  \[
    \frac{q^{(\ell+1)\deg(x)}-1}{q^{(\ell+1)\deg(x)}}
    = 1-q^{-(\ell+1)\deg(x)}
  \]
Independence across all $x\in X$ would give the product in the theorem.
As is the case for Poonen, independence fails when $\deg(x)$ is large relative to $d$ and the difficult part of the proof is showing the error term vanishes as $d$ approaches $\infty$. For this, we need $\ell$ at least as large as the dimension $m$ of $X$.

Regarding \cref{curve-q}, we will define a suitable sheaf \(\sheaf{Q}\) in \cref{curve-ex} whose fiber at a closed point \(x\) is the cotangent space of \(C_x\) at \(x\).

Following Poonen, we will prove \cref{subsheaf-bertini} as a special case of the following more general theorem that allows one to prescribe the first few Taylor expansions at finitely many points.

\begin{maintheorem} \label{subsheaf-TC-finite}
Let \(X\) be a quasiprojective subscheme of \(\bP^n_{\bF_q}\) and \(Z\) a finite subscheme of \(\bP^n\). Fix a subset \(T\subseteq H^0(Z,\sheaf{O}_Z)\). On each connected component \(Z_i\) of \(Z\), fix a nonvanishing coordinate \(x_{j_i}\). For \(f\in S_d\), write \(\restr{f}{Z}\) for the element of \(H^0(Z,\sheaf{O}_Z)\) that on each \(Z_i\) equals the restriction of \(x_{j_i}^{-d}f\) to \(Z_i\).

Assume \(U:=X-(Z\cap X)\) has dimension \(m\) with locally closed embedding \(\iota:U\into\bP^n\). For a locally free quotient \(\sheaf{Q}\) of \(\iota^*\Omega^1_{\bP^n}\) of rank \(\ell\geq m\), define \(\sheaf{E}_d\) and \(\calT_d\) as in \cref{subsheaf-bertini}.

Define
  \[
    \calP_d := \{f\in S_d\mid f \text{ satisfies \(\calT_d\) at all closed \(x\in\bP^n-Z\) and \(\restr{f}{Z}\in T\)}\}.
  \]
Then
  \[
    \lim_{d\to\infty} \Prob(f\in \calP_d)
    = \frac{\# T}{\# H^0(Z,\sheaf{O}_Z)} \zeta_U(\ell+1)^{-1}.
  \]
\end{maintheorem}

The proof is an adaptation of Poonen's original proof; the main innovation is observing that the Taylor condition parameterized by \(X\) need not have anything to do with properties of \(X\).

Again following Poonen, we will prove a stronger version of \cref{subsheaf-TC-finite} that allows us to impose Taylor conditions of arbitrary order at infinitely many points so long as the conditions are no stronger than nonvanishing in locally free quotients of the sheaf of principal parts relative to a finite set of varieties.

\begin{maintheorem} \label{subsheaf-TC-infinite}
Let \(X_1,\dots,X_u\) be quasiprojective subschemes of \(\bP^n_{\bF_q}\) of dimensions \(\dim X_i=m_i\) with locally closed embeddings \(\iota_1,\dots,\iota_u\), respectively.
For each \(i\), let \(\sheaf{Q}_i\) be a locally free quotient of \(\iota_i^*\Omega^1_{\bP^n}\) of rank \(\ell_i\geq m_i\).
Define the sheaves \(\sheaf{E}_{i,d}\) and Taylor conditions \(\calT_{i,d}\) as in \cref{subsheaf-bertini}.

For each closed point \(x\in\bP^n\), fix a positive integer \(M_x\), a nonvanishing coordinate \(x_j\), and a subset \(A_x\subseteq \sheaf{O}_{\bP^n,x}/\frakm_x^{M_x}\).
For \(f\in S_d\), write \(\restr{f}{x}\) for the image of \(x_j^{-d}f\) in \(\sheaf{O}_{\bP^n,x}/\frakm_x^{M_x}\).
Assume that the sets \(A_x\) have been chosen so that for all but finitely many \(x\), \(\restr{f}{x}\in A_x\) whenever \(f\in S_d\) satisfies \(\calT_{i,d}\) at \(x\) for all \(i\).

Define
  \[
    \calP_d := \{f\in S_d\mid \text{\(\restr{f}{x}\in A_x\) for all closed \(x\in\bP^n\)}\}.
  \]
Then
  \[
    \lim_{d\to\infty} \Prob(f\in \calP_d)
    = \prod_{\textup{closed } x\in X} \frac{\# A_x}{\# \sheaf{O}_{\bP^n,x}/\frakm_x^{M_x}}.
  \]
\end{maintheorem}

\subsection{Motivation}

Theorems \ref{subsheaf-bertini}, \ref{subsheaf-TC-finite}, and \ref{subsheaf-TC-infinite} are motivated by the more general Taylor conditions considered by \cite{bilu-howe} in the motivic setting, i.e., in the Grothendieck ring of varieties.
There the authors ask if an arithmetic analog of the following theorem holds over \(\bF_q\) (see the paper for notation):

\begin{theorem*}[{\cite[Theorem B]{bilu-howe}}]
Fix \(f:X\to S\), a proper map of varieties over a field \(K\), \(\sheaf{F}\) a coherent sheaf on \(X\), \(\sheaf{L}\) a relatively ample line bundle on \(X\), and \(r,M\geq 0\).
Then, there is an \(\epsilon>0\) such that as \(T\) ranges over all \(r\)-infinitesimal Taylor conditions on \(\sheaf{F}(d)=\sheaf{F}\otimes\sheaf{L}^d\) with \(M\)-admissible complement,
  \[
    \frac{[\bV(f_*\sheaf{F}(d))^{T\text{-}\,\mathrm{everywhere}}]}{[\bV(f_*\sheaf{F}(d))]}
    = \prod_{x\in X/S} \restr[bigg]{\biggl(1-\frac{[T^c]_x}{[\bV(\PP^r_{\!/S}\sheaf{F}(d))]_x}t\biggr)}{t=1} + O(\bL^{-\epsilon d})
  \]
in \(\,\widehat{\!\widetilde{\mathcal{M}}}_X\).
\end{theorem*}

For Bilu and Howe, a Taylor condition is just a constructible subset of the sheaf of principal parts (viewed as a scheme) and the \(M\)-admissible condition ensures the motivic Euler product converges.
In the arithmetic setting, we also need a good notion of ``admissibility'' for a Taylor condition such that the probability that the condition is satisfied everywhere factors into the local probabilities at closed points.
A counterexample to the most general such Taylor conditions is given in \cref{diag}, suggesting more structure, possibly algebraic as in \cref{subsheaf-bertini}, is necessary.

\subsection{Organization}

In \cref{notation} we set up our notation and give some properties of the sheaf of principal parts.
\cref{counterex} contains a counterexample for the most general Taylor conditions.
In \cref{proof-sec} we prove Theorems \ref{subsheaf-bertini}, \ref{subsheaf-TC-finite}, and \ref{subsheaf-TC-infinite}, and in \cref{applications} we give some applications.

\subsection{Acknowledgments}

We thank the author's PhD advisor, Sean Howe, for numerous useful conversations that lead to the formulation and proof of the theorems above.
The author was partially supported during the preparation of this work by the University of Utah's NSF Research Training Grant \href{https://www.nsf.gov/awardsearch/showAward?AWD_ID=1840190}{\#1840190}.

\section{Notation and definitions} \label{notation}

Throughout, let \(q\) be a power of a prime \(p\) and \(\bF_q\) the field with \(q\) elements.
Let \(S=\bF_q[x_0,\dots,x_n]\) and identify \(S_d:=H^0(\bP_{\bF_q}^n,\sheaf{O}(d))\) with degree \(d\) homogeneous polynomials in \(S\).
Let \(A=\bF_q[x_1,\dots,x_n]\) and \(A_{\leq d}\) the polynomials in \(A\) of degree at most \(d\).

\begin{definition}
Let \(\sheaf{F}\) be a coherent sheaf on a proper \(\bF_q\)-scheme \(X\). An \emph{\(r\)-infinitesimal Taylor condition on \(\sheaf{F}\) at a closed point \(x \in X\)} is a subset
  \[
    \calT_x
    \subseteq \sheaf{F}_x \otimes_{\sheaf{O}_{X,x}} \sheaf{O}_{X,x}/\frakm_x^{r+1}
    =: \restr{\sheaf{F}}{x^{(r)}}.
  \]
An \emph{\(r\)-infinitesimal Taylor condition \(\calT\) on \(\sheaf{F}\)} is a choice of an \(r\)-infinitesimal Taylor condition \(\calT_x\) at \(x\) on \(\sheaf{F}\) for each closed point \(x\). 

We say that a global section \(s\in H^0(X,\sheaf{F})\) \emph{satisfies \(\calT\) at \(x\in X\)} if its image in \(\restr{\sheaf{F}}{x^{(r)}}\) lies in \(\calT_x\), and \emph{satisfies \(\calT\)} if it satisfies \(\calT\) at every closed point \(x\in X\).
\end{definition}

\begin{definition}
Let \(\sheaf{F}\) be a coherent sheaf on a proper \(\bF_q\)-scheme \(X\).
For a subset \(\calP\) of the finite dimensional \(\bF_q\)-vector space \(H^0(X,\sheaf{F})\), denote by \(\Prob(s\in \calP)\) the probability that a random uniformly distributed global section \(s\) of \(\sheaf{F}\) belongs to \(\calP\), i.e.,
  \[
    \Prob(s\in \calP) := \frac{\# \calP}{\# H^0(X,\sheaf{F})}.
  \]
\end{definition}

\begin{remark}
The definition above differs from that of \cite{erman-wood}.
When they write \(\Prob(s\in \calP)\), they mean (in our notation) \(\lim_{d\to\infty}\Prob(s_d\in \calP_d)\) where for each \(d\geq 0\), \(\calP_d\subseteq H^0(X,\sheaf{F}(d))\) and \(s_d\) is a uniform random global section of \(\sheaf{F}(d)\).
\end{remark}

\begin{remark}
Our definition of a Taylor condition assumes \(X\) is proper over \(\bF_q\) so that \(H^0(X,\sheaf{F})\) is a finitely generated \(\bF_q\)-vector space.
This does not contradict allowing quasiprojective \(X\) in \cref{subsheaf-bertini} since there, the Taylor condition is actually on \(\bP^n\).
\end{remark}

\begin{remark}
While we define $r$-infinitesimal Taylor conditions above, our results only concern $1$-infinitesimal Taylor conditions. We give the more general definition to suggest the necessary setup for an arithmetic version of \cite[Theorem B]{bilu-howe} over $\bF_q$. Our results are a first step in that direction.
\end{remark}

\subsection{Sheaves of principal parts} \label{sheaf-PP}

We recall the definition of the sheaf of principal parts and collect some of its relevant properties.

\begin{definition}
Let \(X\to S\) be a morphism of schemes and \(\sheaf{F}\) an \(\sheaf{O}_X\)-module.
Let \(\Delta^{(r)}\) be the \(r\)-th infinitesimal neighborhood of the diagonal \(\Delta\) in \(X\times_S X\) and let \(\delta^{(r)}:\Delta^{(r)}\to X\times_S X\) be the canonical morphism.
Denote by \(\pi_1,\pi_2:X\times_S X\to X\) the corresponding projections and set \(p=\pi_1\circ \delta^{(r)}\) and \(q=\pi_2\circ \delta^{(r)}\).
The \emph{sheaf of \(r\)-th order principal parts of \(\sheaf{F}\) on \(X\) over \(S\)} is
  \[
    \PP_{X/S}^r(\sheaf{F}) := p_*(q^*\sheaf{F}).
  \]
By definition this is an \(\sheaf{O}_X\)-module.
If \(S\) is clear from context, we write \(\PP_X^r(\sheaf{F})\) for \(\PP_{X/S}^r(\sheaf{F})\); if \(X\) is also clear, we write \(\PP^r(\sheaf{F})\).
\end{definition}

References given below are not necessarily the original source of the result.

\begin{lemma}[{\cite[Proposition 16.7.3]{egaIVpartIV}}]
If \(\sheaf{F}\) is quasi-coherent (resp. coherent, of finite type, of finite presentation), then \(\PP^r_{X/S}(\sheaf{F})\) is quasi-coherent (resp. coherent, of finite type, of finite presentation).
\end{lemma}

\begin{lemma}[{\cite[Corollary 16.4.12]{egaIVpartIV} and \cite[III, Lemma 2.1 and Proposition 2.2]{bennett}}] \label{PP-fiber}
If \(S=\Spec k\) for \(k\) a field, \(\sheaf{F}\) is quasi-coherent, and \(x\in X\) is rational over \(k\), then the fiber \(\restr{\PP_{X/S}^r(\sheaf{F})}{x}=\PP_{X/S}^r(\sheaf{F})_x\otimes_{\sheaf{O}_{X,x}}\kappa(x)\) is canonically isomorphic to \(\sheaf{F}_{X,x}/\frakm_x^{r+1}\).

If \(k\) is perfect, then the same is true for any closed point \(x\in X\).
\end{lemma}

\begin{remark}
In our notation, \cref{PP-fiber} says that an \(r\)-infinitesimal Taylor condition on \(\sheaf{F}\) is just a choice of subset of the fiber of \(\PP^r_{X/k}(\sheaf{F})\) for every closed \(x\in X\).
\end{remark}

\begin{lemma}[{\cite[A, Proposition 3.4]{curves:grass}}] \label{PP-exact}
If \(X\to S\) is differentially smooth (see \cite[16.10]{egaIVpartIV}), and \(\sheaf{F}\) is locally free on \(X\), then there is an exact sequence of \(\sheaf{O}_X\)-modules
\begin{equation}
\label{eq:seqparts}
  \begin{tikzcd}[column sep=small]
    0 \ar[r] &
	\Sym^r_{\sheaf{O}_X}(\Omega^1_{X/S})\otimes_{\sheaf{O}_X}\sheaf{F} \ar[r] &
	\PP^r_{X/S}(\sheaf{F}) \ar[r] &
	\PP^{r-1}_{X/S}(\sheaf{F}) \ar[r] &
	0.
  \end{tikzcd}
\end{equation}
If \(X,Y\) are smooth \(S\)-schemes, \(f:X\to Y\) is a morphism of \(S\)-schemes, and \(\sheaf{G}\) is locally free on \(Y\), then there is a map of exact sequences of \(\sheaf{O}_X\)-modules
\begin{equation*}
  \begin{tikzcd}[column sep=small]
	0 \ar[r] &
	\Sym^r_{\sheaf{O}_X}(f^*\Omega^1_{Y/S})\otimes_{\sheaf{O}_X}f^*\sheaf{G} \ar[r] \ar[d] &
	f^*\PP^r_{Y/S}(\sheaf{G}) \ar[r] \ar[d] &
	f^*\PP^{r-1}_{Y/S}(\sheaf{G}) \ar[r] \ar[d]&
	0 \\
	0 \ar[r] &
	\Sym^r_{\sheaf{O}_X}(\Omega^1_{X/S})\otimes_{\sheaf{O}_X}f^*\sheaf{G} \ar[r] &
	\PP^r_{X/S}(f^*\sheaf{G}) \ar[r] &
	\PP^{r-1}_{X/S}(f^*\sheaf{G}) \ar[r] &
	0
  \end{tikzcd}
\end{equation*}
\end{lemma}

\begin{corollary}[{\cite[A, Proposition 3.3]{curves:grass}}]
\label{PP-rank}
In the setting of \cref{PP-exact}, if \(\sheaf{F}\) is locally free of rank \(n\), then \(\PP^r_{X/S}(\sheaf{F})\) is locally free of rank \(n\cdot\binom{\dim X+r}{r}\).
\end{corollary}

We will only need the case $r=1$; in that case the exact sequence \cref{eq:seqparts} becomes
\begin{equation}
\label{eq:seqparts:one}
  \begin{tikzcd}[column sep=small]
    0 \ar[r] &
	\Omega^1_{X/S}\otimes_{\sheaf{O}_X}\sheaf{F} \ar[r] &
	\PP^1_{X/S}(\sheaf{F}) \ar[r] &
	\sheaf{F} \ar[r] &
	0.
  \end{tikzcd}
\end{equation}
Furthermore, we will mainly be interested in the case $X=\bP^n_{\bF_q}$, $S=\Spec(\bF_q)$, and $\sheaf{F}=\sheaf{O}_{\bP^n}$. This satisfies the conditions of \cref{PP-exact} since the structure morphism $\bP^n_{\bF_q}\to\Spec(\bF_q)$ is smooth, hence differentially smooth, and \cref{PP-rank} implies that $\PP^1_{\bP^n}$ is locally free of rank $n+1$.

\section{Counterexamples to most general Taylor conditions} \label{counterex}

The following example shows that arbitrary set-theoretic constructions of Taylor conditions even on \(\sheaf{O}_{\bP^n}(d)\), \(d\geq 0\), can produce local probabilities whose product is not the asymptotic global probability of the condition being satisfied.

\begin{example}[Diagonal argument] \label{diag}
Let \(X=\bP^n_{\bF_q}\) and \(\sheaf{F}=\sheaf{O}_{\bP^n}\).
Both the union of global sections \(S_d\) over all \(d\geq 0\) and the set of closed points of \(\bP^n\) are countably infinite; let \(f_1,f_2,\dots\) and \(x_1,x_2,\dots\) be enumerations of them, respectively.
For each \(i\), fix an isomorphism \(\restr{\sheaf{O}_{\bP^n}(d)}{x_i^{(1)}}\cong \restr{\sheaf{O}_{\bP^n}}{x_i^{(1)}}\).
Define a 1-infinitesimal Taylor condition \(\calT_d\) on \(\sheaf{O}_{\bP^n}(d)\) as follows: for each \(i\), identify \(\restr{\sheaf{O}_{\bP^n}(d)}{x_i^{(1)}}\) with \(\restr{\sheaf{O}_{\bP^n}}{x_i^{(1)}}\) under the fixed isomorphism and set \(\calT_{d,x_i}\) to be all of \(\restr{\sheaf{O}_{\bP^n}}{x_i^{(1)}}\) except the Taylor expansion of \(f_i\) (this does not depend on \(d\)).
Then the local probabilities are \(p_{x_i}=1-q^{-(n+1)\deg(x_i)}\) and the product over all closed points is \(\zeta_{\bP^n}(n+1)^{-1}\).

Globally, however, no section \(f\in S_\homog\) can satisfy this Taylor condition.
Indeed, define
  \[
    \calP_d = \{ f\in S_d\mid \text{\(f\) satisfies \(\calT_d\) at all closed \(x\in\bP^n\)} \}.
  \]
By construction, if \(f=f_i\) in our enumeration, then \(\calT_{d,x_i}\) excludes the Taylor expansion of \(f\), so \(f\) fails \(\calT_d\) at \(x_i\). Thus \(\calP_d=\emptyset\) for all \(d\), and
  \[
    \lim_{d\to\infty} \Prob(f\in \calP_d)
    = 0
    \neq \prod_{i=1}^\infty p_{x_i} = \zeta_{\bP^n}(n+1)^{-1}.
  \]
\end{example}

Some algebraic nature to the condition is likely necessary in general.
In \cref{subsheaf-bertini}, this manifests as ``locally free quotients of the sheaf of differentials''.

\section{More general Taylor conditions} \label{proof-sec}

We now use Poonen's method of the closed point sieve to prove our main theorems.
Throughout this section, let notation be as in \cref{subsheaf-TC-finite}.

\subsection{Points of low degree}

The following lemma says that for finitely many closed points, the local probabilities are independent.

\begin{lemma}[Points of low degree] \label{low-deg}
Let \(U_{<e}\) be the closed points of \(U\) of degree less than \(e\).
Define
  \[
    \calP_{d,e}^\low := \{f\in S_d\mid f \text{ satisfies \(\calT_d\) at all \(x\in U_{<e}\) and \(\restr{f}{Z}\in T\)}\}.
  \]
Then
  \[
    \lim_{d\to\infty} \Prob(f\in \calP_{d,e}^\low)
    = \frac{\# T}{\# H^0(Z,\sheaf{O}_Z)} \prod_{x\in U_{<e}}\bigl(1-q^{-(\ell+1)\deg(x)}\bigr).
  \]
\end{lemma}

\begin{proof}
Let \(U_{<e}=\{x_1,\dots,x_s\}\).
By definition, \(f\in S_d\) fails \(\calT_d\) at \(x_i\) for some \(i\in\{1,\dots,s\}\) if and only if it vanishes under the composition
  \[
    S_d
    \to \sheaf{O}_{\bP^n}(d)_{x_i}/\frakm^2_{x_i}
    \to \restr{\sheaf{E}_d}{x_i}
  \]
Thus \(\calP_{d,e}^\low\) consists of the preimage of \(T\times\prod_{i=1}^s (\restr{\sheaf{E}_d}{x_i}-\{0\})\) under the composition
  \[
    S_d
    \to H^0(Z,\sheaf{O}_Z(d)) \times \prod_{i=1}^s \sheaf{O}_{\bP^n}(d)_{x_i}/\frakm^2_{x_i}
    \to H^0(Z,\sheaf{O}_Z) \times \prod_{i=1}^s \restr{\sheaf{E}_d}{x_i}.
  \]
The first map is surjective for \(d\gg 1\) by \cite[Lemma 2.1]{poonen:bertini} and the second since \(\iota^*\PP^1(\sheaf{O}_{\bP^n}(d))\to\sheaf{E}_d\) is surjective and \(H^0(Z,\sheaf{O}_Z(d))\cong H^0(Z,\sheaf{O}_Z)\), so the composition is surjective.

We have a filtration of \(\kappa(x_i)\)-vector spaces \(0\subset \restr{\sheaf{Q}(d)}{x_i}\subset \restr{\sheaf{E}_d}{x_i}\) whose quotients \(\restr{\sheaf{Q}(d)}{x_i}\) and \(\restr{\sheaf{E}_d}{x_i}/\restr{\sheaf{Q}(d)}{x_i}\) have dimensions \(\ell\) and \(1\), respectively, hence \(\restr{\sheaf{E}_d}{x_i}-\{0\}\) has size \(q^{(\ell+1)\deg(x_i)}-1\), and the local probability of vanishing is \(1-q^{-(\ell+1)\deg(x_i)}\).
As this does not depend on \(d\), the result follows.
\end{proof}

\subsection{Points of medium degree}

\begin{lemma}[Points of medium degree] \label{med-deg}
For \(e>0\), define
  \[
    \calQ_{d,e}^\med := \{f\in S_d\mid f \text{ fails \(\calT_d\) at some \(x\in U\) with } e\leq\deg(x)\leq \tfrac{d}{\ell+1}\}.
  \]
Then
  \[
    \lim_{e\to\infty}\lim_{d\to\infty}\Prob(f\in\calQ_{d,e}^\med) = 0.
  \]
\end{lemma}

\begin{proof}
Let \(x\) be a closed point of \(U\) with \(e\leq \deg(x)\leq \frac{d}{\ell+1}\). We have \(\dim_{\bF_q}\restr{\sheaf{E}_d}{x} = (\ell+1)\deg(x)\leq d\) by assumption.
Note the argument in \cite[Lemma 2.1]{poonen:bertini} works exactly the same here with the map \(S_d\to\restr{\sheaf{E}_d}{x}\), so this map is surjective and identical reasoning as in \cite[Lemma 2.3]{poonen:bertini} shows the fraction of \(f\in S_d\) that vanish in \(\restr{\sheaf{E}_d}{x}\) is \(q^{-(\ell+1)\deg(x)}\).

Now we follow Poonen's proof of \cite[Lemma 2.4]{poonen:bertini}.
By \cite{lang-weil}, there is a constant \(c>0\) depending only on \(U\) such that \(\# U(\bF_{q^{r}})\leq cq^{rm}\).
With the result above, this gives
\begin{align*}
  \Prob(f\in \calQ_{d,e}^\med) &\leq \sum_{r=e}^{\floor{d/(\ell+1)}}(\text{\# of points of degree \(r\)})\cdot q^{-(\ell+1)r} \\
  &\leq \sum_{r=e}^{\floor{d/(\ell+1)}} \# U(\bF_{q^{r}})\cdot q^{-(\ell+1)r} \\
  &\leq \sum_{r=e}^\infty cq^{rm}q^{-(\ell+1)r}
\end{align*}
Since \(\ell\geq m\), this converges to \(\frac{cq^{e(m-\ell-1)}}{1-q^{m-\ell-1}}\). This is independent of \(d\) and goes to zero as \(e\) goes to \(\infty\).
\end{proof}

\subsection{Points of high degree}

As usual with proofs using the closed point sieve, showing the contribution from high degree points is negligible is the hardest part of the proof.

\begin{lemma}[Points of high degree] \label{high-deg}
Define
  \[
    \calQ_d^\high := \{f\in S_d\mid f \text{ fails \(\calT_d\) at some \(x\in U_{>\frac{d}{\ell+1}}\)}\}.
  \]
Then \(\lim_{d\to\infty} \Prob(f\in\calQ_d^\high)=0\).
\end{lemma}

\begin{proof}
As in \cite[Lemma 2.6]{poonen:bertini}, we reduce to the affine case \(\iota:U\into\bA^n\), also dehomogenizing to identify \(S_d\) with \(A_{\leq d}\).

Pulling back the exact sequence \cref{eq:seqparts:one}, we have the commutative diagram
  \[
    \begin{tikzcd}[row sep=scriptsize,column sep=small]
      0 \ar[r] & \iota^*\Omega_{\bA^n/\bF_q}^1 \ar[r,hookrightarrow]
      \ar[d] & \iota^*\PP^1(\sheaf{O}_{\bA^n}) \ar[r] \ar[d] & \sheaf{O}_U \ar[r] & 0 \\
      0 \ar[r] & \sheaf{Q} \ar[r,hookrightarrow] & \sheaf{E} & &
    \end{tikzcd}
  \]
Given a closed point \(x\in U\), the map on fibers \(\restr{\iota^*\Omega_{\bA^n/\bF_q}^1}{x}\to\restr{\sheaf{Q}}{x}\) is surjective, and every element of \(\restr{\iota^*\Omega_{\bA^n/\bF_q}^1}{x}\cong \frakm_x/\frakm_x^2\) is the restriction of some \(dt\) to \(x\) where \(t\) is some element of \(A\).
Choose \(t_1,\dots,t_\ell\in A\) such that the restrictions of \(dt_1,\dots,dt_\ell\) to \(x\) map to a \(\kappa(x)\)-basis of \(\restr{\sheaf{Q}}{x}\).
By Nakayama's lemma, the elements \(dt_1,\dots,dt_\ell\) in the stalk \((\iota^*\Omega_{\bA^n/\bF_q}^1)_x\) map to an \(\sheaf{O}_{U,x}\)-basis for \(\sheaf{Q}_x\).
Call this basis \(Q_1,\dots,Q_\ell\) and let \(\partial_1,\dots,\partial_\ell\) be the corresponding dual basis.

Now we mimic the proof of \cite[Lemma 2.6]{poonen:bertini}.

We have
\begin{equation*}
  \begin{aligned}
    \Hom_{\sheaf{O}_{U,x}}(\sheaf{Q}_x,\sheaf{O}_{U,x})
    &\subset \Hom_{\sheaf{O}_{U,x}}((\iota^*\Omega^1_{\bA^n})_x,\sheaf{O}_{U,x}) \\
    &= \Hom_{\sheaf{O}_{\bA^n,x}}(\Omega^1_{\bA^n,x},\sheaf{O}_{U,x}) \\
    &= \Der_{\bF_q}(\sheaf{O}_{\bA^n,x},\sheaf{O}_{U,x})
  \end{aligned}
\end{equation*}
where the first inclusion follows since \( \Hom_{\sheaf{O}_{U,x}}(-,\sheaf{O}_{U,x})\) is contravariant left exact.
Thus we can think of the dual basis elements \(\partial_i\) as \(\bF_q\)-derivations \(\sheaf{O}_{\bA^n,x}\to\sheaf{O}_{U,x}\).
Let $I(U)\subset A$ be the ideal of polynomials vanishing on $U$.
Choose \(s\in A/I(U)\) with \(s(x)\neq 0\) to clear denominators so \(D_i=s\partial_i\) is a global derivation \(A\to A/I(U)\).
We can find a neighborhood \(N_x\) of \(x\) on which \(Q_1,\dots,Q_\ell\) generate \(\sheaf{Q}\) and such that \(s\in\sheaf{O}_U(N_x)^\times\).
As we can cover \(U\) with finitely many such \(N_x\), we may assume \(U\subset N_x\), and that the \(Q_1,\dots,Q_\ell\) generate \(\sheaf{Q}\) globally.

Set \(\tau=\max_i\{\deg t_i\}\), \(\gamma=\floor{(d-\tau)/p}\), and \(\eta=\floor{d/p}\).
If \(f_0\in A_{\leq d}\), \(g_1,\dots,g_\ell\in A_{\leq\gamma}\), and \(h\in A_{\leq\eta}\) are selected uniformly at random, then the distribution of
  \[
    f = f_0+g_1^pt_1+\dots+g_\ell^pt_\ell+h^p
  \]
is uniform over \(A_{\leq d}\).
We will bound the probability that for such an \(f\), there is a closed point \(y\in U_{>d/(\ell+1)}\) where \(f\) is zero in the fiber of \(\sheaf{E}\) at \(y\).
Let \(1\) be the constant function in \(\iota^*\PP^1(\sheaf{O}_{\bA^n})\), and \(R\) its image in \(\sheaf{E}\).
Then \(R,Q_1,\dots,Q_\ell\) are a basis for \(\sheaf{E}\), giving a trivialization \(\sheaf{E}\cong \sheaf{O}_U^{\ell+1}\).
In this trivialization, the map sending a polynomial \(f\) to its first order Taylor expansion in \(\iota^*\PP^1(\sheaf{O}_{\bA^n})\) then to \(\sheaf{E}\) is given by \((f,\partial_1 f,\dots,\partial_\ell f)\).
Thus \(f\) is zero in \(\restr{\sheaf{E}}{y}\) if and only if \(f(y)=(D_1 f)(y)=\dots=(D_\ell f)(y)=0\).

Since \(\operatorname{char}\bF_q=p\), we have
\begin{align*}
  D_i f &= D_i f_0 + g_1^p D_i t_1 + t_1 pD_i g_1 + \dots + g_\ell^p D_i t_\ell + t_\ell pD_i g_\ell + pD h \\
  &= D_i f_0+g_i^ps
\end{align*}
for \(i=1,\dots,\ell\). By abuse of notation we will consider the \(D_if\) as defining hypersurfaces in \(\bA^n\) by choosing a lift to \(A\) of minimal degree.
Define
  \[
    W_i = U \cap \{D_1f=\dots=D_if=0\}.
  \]
\begin{claim}
For \(0\leq i\leq \ell-1\), conditioned on a choice of \(f_0,g_1,\dots,g_i\) such that \(\dim W_i\leq \max(m-i,0)\), the probability that \(\dim W_{i+1}\leq \max(m-i-1,0)\) is \(1-o(1)\) as \(d\to\infty\). (The function of \(d\) represented by \(o(1)\) depends on \(U\) and the \(D_i\).)
\end{claim}

The statement is clear when $i>m$ so we may assume $\ell=m$.

Let \(V_1,\dots,V_e\) be the \((m-i)\)-dimensional irreducible components of \((W_i)_\red\).
By Bézout's theorem,
  \[
    e
    \leq (\deg\,\overline{U})(\deg D_1f)\,\cdots\,(\deg D_if)
    = O(d^i)
  \]
as \(d\to\infty\), where \(\,\overline{U}\) is the projective closure of \(U\).
As \(\dim V_k\geq 1\), there exists a coordinate \(x_j\), depending on \(k\), such that the projection \(x_j(V_k)\) has dimension 1.

We want to bound the set
  \[
    G_k^{\mathrm{bad}}
    := \{g_{i+1}\in A_{\leq\gamma} \mid D_{i+1} f=D_{i+1}f_0+g_{i+1}^ps\text{ vanishes identically on \(V_k\)}\}
  \]
since for any \(g_{i+1}\in G_k^{\mathrm{bad}}\), \(V_k\subset W_{i+1}\) and then \(\dim W_{i+1}\) would fail to be \(\leq m-i-1\).

If \(g,g'\in G_k^{\mathrm{bad}}\), then on \(V_k\),
\begin{align*}
  0 &= \frac{g^ps-g'^ps}{s} \\
  &= g^p-g'^p \\
  &= (g-g')^p
\end{align*}
so if \(G_k^{\mathrm{bad}}\) is nonempty, it is a coset of the subspace of functions in \(A_{\leq\gamma}\) that vanish on \(V_k\).
The codimension of that subspace is at least \(\gamma+1\) since a nonzero polynomial in \(x_j\) does not vanish on \(V_k\).
Thus the probability that \(D_{i+1}f\) vanishes on some \(V_k\) is at most \(e q^{-(\gamma+1)}=o(1)\) as \(d\to\infty\).

\begin{claim}
Conditioned on a choice of \(f_0,g_1,\dots,g_\ell\) for which \(W_\ell\) is finite, \(\Prob(H_f\cap W_\ell\cap U_{>d/(\ell+1)}=\emptyset)=1-o(1)\) as \(d\to\infty\).
\end{claim}

In fact, we need only show this for \(H_f\cap W_m\cap U_{>d/(\ell+1)}\). 
The same Bézout argument as above shows \(\#W_m\) is \(O(d^m)\).
For a given \(y\in W_m\), the set \(H^{\mathrm{bad}}\) of \(h\in A_{\leq\eta}\) for which \(H_f\) passes through \(y\) is either empty or a coset of \(\ker(\eval_y:A_{\leq\eta}\to\kappa(y))\).

If \(\deg(y)>\frac{d}{\ell+1}\), then \cite[Lemma 2.5]{poonen:bertini} implies \(\frac{\#H^{\mathrm{bad}}}{\#A_{\leq\eta}}\leq q^{-\nu}\) where \(\nu=\min(\eta+1,\frac{d}{\ell+1})\).
Hence
  \[
    \Prob(H_f\cap W_m\cap U_{>d/(\ell+1)}\neq\emptyset)
    \leq \#W_mq^{-\nu}=O(d^mq^{-\nu})
  \]
which by assumption is \(o(1)\) as \(d\to\infty\).

Given the two claims, we have
\begin{multline*}
  \lim_{d\to\infty}\Prob\bigl(\dim W_i=m-i \text{ for all \(1\leq i\leq\ell\) and } H_f\cap W_m\cap U_{>d/(\ell+1)}=\emptyset\bigr) \\
  \begin{aligned}
    &= \prod_{i=0}^{m-1} (1-o(1))\cdot (1-o(1)) \\
    &= 1-o(1).
  \end{aligned}
\end{multline*}
So the same holds for \(W_\ell\).
But now \(H_f\cap W_\ell\) is the subvariety of \(U\) defined by failing \(\calT_d\), so \(H_f\cap W_\ell\cap U_{>d/(\ell+1)}\) is the set of points of degree \(>\frac{d}{\ell+1}\) where \(H_f\cap U\) fails \(\calT_d\).
\end{proof}

\subsection{Proofs of Theorems \ref*{subsheaf-bertini}, \ref*{subsheaf-TC-finite}, and \ref*{subsheaf-TC-infinite}}

\begin{proof}[Proof of \cref{subsheaf-TC-finite}]
We have
  \[
    \calP_d\subseteq \calP_{d,e}^\low
    \subseteq \calP_d\cup \calQ_{d,e}^\med\cup\calQ_d^\high
  \]
so
\begin{align*}
  \Prob(s\in \calP_{d,e}^\low)
  &\geq \Prob(s\in\calP_d) \\
  &\geq \Prob(s\in \calP_{d,e}^\low) - \Prob(s\in \calQ_{d,e}^\med) - \Prob(s\in \calQ_d^\high).
\end{align*}
By \cref{low-deg,med-deg,high-deg}, letting \(d\), then \(e\) go to \(\infty\) gives the result.
\end{proof}

\begin{proof}[Proof of \cref{subsheaf-bertini}]
Take \(Z=\emptyset\) in \cref{subsheaf-TC-finite}.
\end{proof}

\begin{proof}[Proof of \cref{subsheaf-TC-infinite}]
The reasoning here is the same as in the proof of \cite[Theorem 1.3]{poonen:bertini}. Given a condition on sections no stronger than nonvanishing in the fiber of a single \(\sheaf{E}_{i,d}\) at all except finitely many points, the probability of failing this condition goes to \(0\) as \(d\to\infty\) by \cref{med-deg,high-deg}. Now considering a condition no stronger than nonvanishing in the fiber of each of \(\sheaf{E}_{1,d},\dots,\sheaf{E}_{u,d}\) at all except finitely many points, the probability of failing is still zero as this is a finite union of sets with probability zero. Thus we can approximate \(\calP_d\) by the sets \(\calP_{d,e}^\low\) defined by satisfying the condition at points of degree at most \(e\). But now the result follows by \cite[Lemma 2.1]{poonen:bertini} and identical reasoning as in the proof of \cref{low-deg}.
\end{proof}

\section{Applications} \label{applications}

\begin{example}[Poonen's Bertini]
To get \cite[Theorem 1.1]{poonen:bertini}, assume \(X\) is smooth and take \(\sheaf{Q}=\Omega^1_{X/\bF_q}\) in \cref{subsheaf-bertini}. Similarly, \cite[Theorem 1.2]{poonen:bertini} follows from \cref{subsheaf-TC-finite}.

Our \cref{subsheaf-TC-infinite} does not imply \cite[Theorem 1.3]{poonen:bertini} since we don't work with the completed local rings, however it does imply the weaker version where, at each point, one only controls the Taylor expansion up to finitely many terms.
\end{example}

\begin{example} \label{rel-smooth-app}
Let \(X\) be a quasiprojective subscheme of \(\bP^n_{\bF_q}\) of dimension \(m\) with locally closed embedding \(\iota\) and let \(\Delta:X\into X\times_{\bF_q}\bP^n\) be the graph of \(\iota\).
Suppose \(j:Z\into X\times_{\bF_q}\bP^n\) is a closed embedding such that the projection \(\varphi:Z\to X\) is smooth of relative dimension \(\ell\geq m\), and such that \(\Delta\) factors as
  \[
    X \xrightarrow{\alpha} Z \xhookrightarrow{j} X\times\bP^n
  \]
for some morphism \(\alpha:X\to Z\).

We have a surjection of sheaves
  \[
    \Omega^1_{X\times\bP^n/X}\onto j_*\Omega^1_{Z/X}
  \]
which induces a surjection
  \[
    \Delta^*\Omega^1_{X\times\bP^n/X}\onto \Delta^*j_*\Omega^1_{Z/X}.
  \]
The left side is isomorphic to \(\iota^*\Omega^1_{\bP^n}\); indeed, let \(p:X\times\bP^n\to\bP^n\) be projection onto the second coordinate.
Then by standard base change for the sheaf of differentials,
\begin{align*}
  \Delta^*\Omega^1_{X\times\bP^n/X} &\cong \Delta^*p^*\Omega^1_{\bP^n/\bF_q} \\
  &= (p\circ\Delta)^*\Omega^1_{\bP^n/\bF_q} \\
  &= \iota^*\Omega^1_{\bP^n/\bF_q}.
\end{align*}
Define \(\sheaf{Q}=\Delta^*j_*\Omega^1_{Z/X}\).
This is locally free: by assumption, \(\Delta=j\circ\alpha\) so \(\sheaf{Q}=\alpha^*j^*j_*\Omega^1_{Z/X}\cong \alpha^*\Omega^1_{Z/X}\).
As \(\varphi\) is smooth of relative dimension \(\ell\), \(\Omega^1_{Z/X}\) is locally free of rank \(\ell\) and thus so is \(\sheaf{Q}\).

With \(\sheaf{Q}\) as above, define \(\sheaf{E}_d\), \(\calT_d\), and \(\calP_d\) as in \cref{subsheaf-bertini}.
Applying the theorem, we get
  \[
    \lim_{d\to\infty} \Prob(f\in \calP_d) = \zeta_X(\ell+1)^{-1}.
  \]
\end{example}

\begin{example} \label{curve-ex}
We now answer \cref{curve-q} as a specific instance of \cref{rel-smooth-app}.
Assume \(\operatorname{char}(\bF_q)\neq 2\).
Choose a finite, reduced, degree 4 subscheme \(Y\) of \(\bP^2_{\bF_q}\) whose points are geometrically in general position.
Let \(\iota:X\into \bP^2_{\bF_q}\) be a curve whose geometric points are in general position with the points of \(Y\).
Then for each closed point \(x\in X\), there is a unique smooth conic \(C_x\) (defined over \(\kappa(x)\)) passing through \(x\) and each point of \(Y\). Let \(j:C\into X\times\bP^2\) be the inclusion of the subscheme \(C\) parameterizing the data \(\{(x,y)\mid x\in X,y\in C_x\}\).
Then \(\Delta\) factors as \(j\circ\alpha\) where \(\alpha\) is the diagonal into \(C\), and \(\varphi:C\to X\) is smooth of relative dimension \(1\), so the conditions of the example are satisfied.

Let \(f\in S_d\).
With \(\sheaf{Q}\) defined as in \cref{rel-smooth-app}, the hypersurface \(H_f\) intersects \(C_x\) transversely at \(x\) if and only if it does not vanish in the fiber of \(\sheaf{Q}\) at \(x\).
Thus the example above shows the probability that a random plane curve intersects \(C_x\) transversely at \(x\) for all closed \(x\in X\) is \(\zeta_X(2)^{-1}\).
\end{example}

\begin{example}
Let \(L\) be the line at infinity in \(\bP^2_{\bF_q}\); write the homogeneous coordinates on \(\bP^2\) as \(x_0,x_1,x_2\). In the affine chart \(x_0\neq 0\), choose \(\bF_q\)-points \(P_1=(0,0)\), \(P_2=(0,1)\), \(P_3=(1,0)\), and \(P_4=(1,1)\).  Then the lines through pairs of points in \(P_1,P_2,P_3,P_4\) intersect \(L\) in four points; set \(U\) to be \(L\) with these four points removed and \(Y:=\{P_1,P_2,P_3,P_4\}\). Define \(C_x\) as above. By \cref{curve-ex}, the probability that a random plane curve intersects \(C_x\) transversely at \(x\) for all \(x\in U\) is \(\zeta_U(2)^{-1}\). Recall that for a scheme \(X\) of finite type over \(\bF_q\) with closed subscheme \(Z\), we have \(\zeta_{X\setminus Z}(s)=\zeta_X(s)/\zeta_Z(s)\). Writing \(Z=\{Q_1,Q_2,Q_3,Q_4\}\) for the set of four points removed from \(L\), this implies
  \[
    \zeta_U(2)^{-1}
    = \frac{\zeta_Z(2)}{\zeta_L(2)}
    = \frac{(1-q^{-2})^{-4}}{\frac{1}{(1-q^{-2})(1-q^{-1})}}
    = \frac{1-q^{-1}}{(1-q^{-2})^3}.
  \]
\end{example}

\bibliographystyle{amsalpha}
\bibliography{axiomatic-bertini}

\end{document}